\documentclass[reqno]{amsart}
\usepackage{color}

\newtheorem{theorem}{Theorem}
\newtheorem{lemma}[theorem]{Lemma}

\theoremstyle{definition}

\theoremstyle{remark}

\numberwithin{equation}{section}

\newcommand{\intav}[1]{\mathchoice {\mathop{\vrule width 6pt height 3 pt depth  -2.5pt
\kern -8pt \intop}\nolimits_{\kern -6pt#1}} {\mathop{\vrule width
5pt height 3  pt depth -2.6pt \kern -6pt \intop}\nolimits_{#1}}
{\mathop{\vrule width 5pt height 3 pt depth -2.6pt \kern -6pt
\intop}\nolimits_{#1}} {\mathop{\vrule width 5pt height 3 pt depth
-2.6pt \kern -6pt \intop}\nolimits_{#1}}}

\newcommand{\intavl}[1]{\mathchoice {\mathop{\vrule width 6pt height 3 pt depth  -2.5pt
\kern -8pt \intop}\limits_{\kern -6pt#1}} {\mathop{\vrule width 5pt
height 3  pt depth -2.6pt \kern -6pt \intop}\nolimits_{#1}}
{\mathop{\vrule width 5pt height 3 pt depth -2.6pt \kern -6pt
\intop}\nolimits_{#1}} {\mathop{\vrule width 5pt height 3 pt depth
-2.6pt \kern -6pt \intop}\nolimits_{#1}}}

 \newcommand{\mc}{\mathcal}

 \newcommand{\R}{\mathbb{R}}

 \newcommand{\Z}{\mathbb{Z}}

 \newcommand{\hh}{\tfrac12}

 \newcommand{\dy}{\text{\rm d}y}

\begin{document}

\title[Discrete maximal operators]{On the endpoint regularity of discrete \\ maximal operators}

\author[Carneiro]{Emanuel Carneiro}
\address{IMPA - Instituto de Matem\'{a}tica Pura e Aplicada, Estrada Dona Castorina, 110, Rio de Janeiro, Brazil 22460-320.}
\email{carneiro@impa.br}

\author[Hughes]{Kevin Hughes}
\address{Department of Mathematics, Princeton University, Fine Hall, Washington Road, Princeton, NJ, 08544}
\email{kjhughes@math.princeton.edu}


\date{June 30, 2012}


\maketitle

\centerline{\it Dedicated to Professor William Beckner on the occasion of his 70th birthday.}

\vskip 0.2 in

\begin{abstract} Given a discrete function $f:\Z^d \to \R$ we consider the maximal operator
$$Mf(\vec{n}) = \sup_{r\geq0} \frac{1}{N(r)} \sum_{\vec{m} \in \overline{\Omega}_r} \big|f(\vec{n} + \vec{m})\big|,$$
where $\big\{\overline{\Omega}_r\big\}_{r \geq 0}$ are dilations of a convex set $\Omega$ (open, bounded and with Lipschitz boudary) containing the origin and $N(r)$ is the number of lattice points inside $\overline{\Omega}_r$. We prove here that the operator $f \mapsto \nabla M f$ is bounded and continuous from $l^1(\Z^d)$ to $l^1(\Z^d)$. We also prove the same result for the non-centered version of this discrete maximal operator.

\end{abstract}

\section{Introduction}
\subsection{Background} For a function $f \in L^1_{loc}(\R^d)$ the (centered) Hardy-Littlewood maximal operator is defined as
\begin{equation*}
Mf(x) = \sup_{r>0} \frac{1}{m(B_r)} \int_{B_r} |f(x + y)|\,\dy,
\end{equation*}
where $B_r$ is the ball of radius $r$ centered at the origin and $m(B_r)$ is the $d$-dimensional Lebesgue measure of this ball. A basic result in harmonic analysis is that $M:L^p(\R^d) \to L^p(\R^d)$ is a bounded operator for $p>1$, and that it satisfies a weak-type estimate $M: L^1(\R^d) \to L^1_{weak}(\R^d)$ at the endpoint $p=1$. The same holds in the non-centered case, when we consider the supremum over balls that simply contain the point $x$. In both instances we may also replace the balls by dilations of a convex set with Lipschitz boundary (since these have bounded eccentricity). 

\smallskip

Over the last years several works addressed the problem of understanding the behavior of differentiability under a maximal operator. This program began with Kinnunen \cite{Ki} who investigated the action of the classical Hardy-Littlewood maximal operator in Sobolev spaces and showed that $M:W^{1,p}(\R^d) \to W^{1,p}(\R^d)$ is bounded for $p>1$. This paradigm that an $L^p$-bound implies a $W^{1,p}$-bound was later extended to a local version of the maximal operator \cite{KL}, to a fractional version \cite{KiSa} and to a multilinear version \cite{CM}. The continuity of $M:W^{1, p} \to W^{1, p}$ for $p>1$ was established by Luiro in \cite{Lu1} for the classical Hardy-Littlewood maximal operator and in \cite{Lu2} for its local version. Note that this is a non-trivial problem since we do not have sublinearity for the weak derivatives of the Hardy--Littlewood maximal function. 

\smallskip

Understanding the regularity at the endpoint case seems to be a deeper issue. In this regard, one of the main questions was posed by  Haj\l asz and Onninen in \cite[Question 1]{HO}: {\it is the operator $f \mapsto \nabla Mf$ bounded from $W^{1,1}(\R^d)$ to $L^1(\R^d)$}? Observe that a bound of the type
\begin{equation}\label{eq1.1}
\|\nabla M f\|_{L^1(\R^d)} \leq C \big(\|f\|_{L^1(\R^d)} + \|\nabla f\|_{L^1(\R^d)}\big)
\end{equation}
would imply, via a dilation invariance argument, the bound
\begin{equation}\label{eq1.2}
\|\nabla M f\|_{L^1(\R^d)} \leq C  \|\nabla f\|_{L^1(\R^d)},
\end{equation}
and so the fundamental question would be to compare the variation of $Mf$ with the variation of the original function $f$ (perhaps having the additional information that $f$ is integrable). In the work \cite{Ta}, Tanaka obtained the bound \eqref{eq1.2} in dimension $d=1$ for the {\it non-centered} Hardy-Littlewood maximal operator with constant $C=2$. This was later improved by Aldaz and P\'{e}rez L\'{a}zaro \cite{AP} who obtained \eqref{eq1.2} with the sharp $C=1$ under the minimal assumption that $f$ is of bounded variation (still, only in dimension $d=1$ and for the non-centered maximal operator). The progress in the centered case is very recent and also only in dimension $d = 1$. In \cite{Ku} O. Kurka showed that if $f$ is of bounded variation on $\R$ then
$${\rm Var} (Mf) \leq C \,{\rm Var}(f)$$
for some constant $C > 1$, where ${\rm Var}(f)$ denotes the total variation of the function $f$. This result was later adapted to the one-dimensional discrete setting by Temur \cite{Te}. It is likely that the sharp constant in KurkaÕs inequality should be $C = 1$, but this remains an open problem. Regularity results of similar flavour for the heat flow
maximal operator and the Poisson maximal operator were obtained in \cite{CS}.

\subsection{The discrete analogue} We address here this problem in the discrete setting. We shall generally denote by $\vec{n} = (n_1, n_2, ...., n_d)$ a vector in $\Z^d$ and for a function $f:\Z^d \to \R$ we define its $l^p$-norm as usual:
\begin{equation*}
\|f\|_{l^p(\Z^d)} = \left( \sum_{\vec{n} \in \Z^d} \big|f(\vec{n})\big|^p\right)^{1/p},
\end{equation*}
if $1 \leq p < \infty$, and 
\begin{equation*}
\|f\|_{l^{\infty}(\Z^d)} = \sup_{\vec{n} \in \Z^d}\big|f(\vec{n})\big|.
\end{equation*}
The gradient $\nabla f$ of a discrete function $f$ will be the vector
\begin{equation*}
\nabla f (\vec{n}) = \left(\frac{\partial f}{\partial x_1}(\vec{n}), \frac{\partial f}{\partial x_2}(\vec{n}), ..., \frac{\partial f}{\partial x_d}(\vec{n})\right),
\end{equation*}
where
\begin{equation*}
\frac{\partial f}{\partial x_i}(\vec{n}) := f(\vec{n} + \vec{e}_i) - f(\vec{n}),
\end{equation*}
and $\vec{e}_i = (0,0,...,1,...,0)$ is the canonical $i$-th base vector. 

\smallskip

Now let $\Omega \subset \R^d$ be a bounded open subset that is convex with Lipschitz boundary. Let us assume that  $\vec{0} \in {\rm int}(\Omega)$ and normalize it so that $\vec{e}_d \in \partial\Omega$.  We now define the set that will play the role of the ``ball of center $\vec{x}_0$ and radius $r$" in our maximal operators. For $r >0$ we write
\begin{equation*}
\overline{\Omega}_r (\vec{x}_0)= \big\{ \vec{x} \in \Z^d; \, r^{-1}(\vec{x} - \vec{x}_0) \in \overline{\Omega}\big\},
\end{equation*}
and for $r=0$ we put
\begin{equation*}
\overline{\Omega}_0(\vec{x}_0) = \big\{\vec{x}_0\big\}.
\end{equation*}
Whenever $\vec{x}_0 = \vec{0}$ we shall write  $\overline{\Omega}_r = \overline{\Omega}_r \big(\vec{0}\big)$ for simplicity. For instance, to work with regular $l^p$-balls one should consider $\Omega = \big\{\vec{x} \in \R^d; |\vec{x}|_p <1\big\}$. 

\smallskip

From now on we use the letter $M$ to denote the centered discrete maximal operator associated to $\Omega$ given by
\begin{equation}\label{eq1.3}
Mf(\vec{n}) = \sup_{r\geq0} \frac{1}{N(r)} \sum_{\vec{m} \in \overline{\Omega}_r} \big|f(\vec{n} + \vec{m})\big|,
\end{equation}
where $N(r)$ is the number of lattice points in the set $\overline{\Omega}_r$. We define the non-centered discrete maximal operator $\widetilde{M}$ associated to $\Omega$ in a similar way, by writing
\begin{equation}\label{eq1.4}
\widetilde{M}f(\vec{n}) = \sup_{r\geq0} \frac{1}{N(\vec{x}_0,r)} \sum_{\vec{m} \in \overline{\Omega}_r(\vec{x}_0)} |f(\vec{m})|,
\end{equation}
where the supremum is taken over all ``balls"  $\overline{\Omega}_r(\vec{x}_0)$ such that $\vec{n} \in \overline{\Omega}_r(\vec{x}_0)$, and $N(\vec{x}_0,r)$ denotes the number of lattice points in the set $\overline{\Omega}_r(\vec{x}_0)$. 

\smallskip

These convex $\Omega$-balls have roughly the same behavior as the regular balls, from the geometric and arithmetic points of view.  For instance, we have the following asymptotics \cite[Chapter VI \S 2, Theorem 2]{Lang} for the number of lattice points
\begin{equation}\label{asymp}
N(\vec{x}_0,r)= C_{\Omega}\, r^d + O\big(r^{d-1}\big)
\end{equation}
as $r \to \infty$, where $C_{\Omega} = m(\Omega)$ is the $d$-dimensional volume of $\Omega$, and the constant implicit in the big O notation depends only on the dimension $d$ and on the set $\Omega$ (e.g. if $\Omega$ is the $l^{\infty}$-ball  we have the exact expression $N(r) = (2\lfloor{r}\rfloor +1)^d$). 

\smallskip

As in the continuous case, both $M$ and $\widetilde{M}$ are of strong type $(p,p)$, if $p>1$, and of weak type $(1,1)$ (see for instance \cite[Chapter X]{SteinHA}). It is then natural to ask how the regularity theory transfers from the continuous to the discrete setting. By the triangle inequality one sees that, in the discrete setting,  the Sobolev norm $\|f\|_{l^p} + \|\nabla f\|_{l^p}$ is equivalent to the norm $\|f\|_{l^p}$, and thus the question of whether $M$ and $\widetilde{M}$ are bounded in discrete Sobolev spaces is trivially true for $p>1$. On the other hand, the regularity at the endpoint case $p=1$ is a very interesting topic and the main objective of this paper is to present the folllowing result.

\begin{theorem}[Endpoint regularity of discrete maximal operators] \label{thm1}
Let $d\geq 1$ and consider $M$ and $\widetilde{M}$ as defined in \eqref{eq1.3} and \eqref{eq1.4}.
\begin{enumerate}
\item [(i)] (Centered case) The operator $f \mapsto \nabla Mf$ is bounded and continuous from $l^1\big(\Z^d\big)$ to $l^1\big(\Z^d\big)$.
\item[(ii)] (Non-centered case) The operator $f \mapsto \nabla \widetilde{M}f$ is bounded and continuous from $l^1\big(\Z^d\big)$ to $l^1\big(\Z^d\big)$.
\end{enumerate}
\end{theorem}

The boundedness part in Theorem \ref{thm1} provides a positive answer to the question of  Haj\l asz and Onninen \cite[Question 1]{HO} in the discrete setting, in all dimensions and for this general family of centered or non-centered maximal operators with convex $\Omega$-balls. The insight for this part was originated in a joint work of the authors with J. Bober and L. B. Pierce \cite{BCHP} where the case $d=1$ was treated, and it has two main ingredients: (i) a double counting argument to evaluate the maximal contribution of each point mass of $f$ to $\|\nabla Mf\|_{l^1}$; (ii) a summability argument over the sequence of local maxima and local minima of $Mf$. The technique is now refined to contemplate the $n$-dimensional case and this general family of operators.

\smallskip

The continuity result is a novelty in the endpoint regularity theory. Luiro's framework \cite{Lu1} for the continuity of the classical Hardy-Littlewood maximal operator in the Sobolev space $W^{1,p}(\R^d)$, for $p>1$, is not adaptable since it relies on the $L^p$-boundedness of this operator (which we do not have here), and we will only be able to use a few ingredients of it. The heart of our proof lies instead on the two core ideas mentioned above for the boundedness part and a useful application of the Brezis-Lieb lemma \cite{BL}.

\medskip

\noindent {\it Remark 1:} One might ask if inequality \eqref{eq1.2} holds in the discrete case, which would be a stronger result than our Theorem \ref{thm1}. This has only been proved in dimension $d=1$ for the {\it non-centered} maximal operator (see \cite{BCHP}) with sharp constant $C=1$ (i.e. the non-centered maximal function does not increase the variation of a function). Note that the dilation invariance argument to deduce \eqref{eq1.2} from \eqref{eq1.1} fails in the discrete setting.

\medskip

\noindent{\it Remark 2:} If we consider for instance the one-dimensional discrete centered Hardy-Littlewood maximal operator with regular balls applied to the delta function $f(0)=1$ and $f(n) = 0$ for $n\neq0$, we obtain $Mf(n) = 1/(2|n|+1)$ and thus $(Mf)'(n) = O\big(|n|^{-2}\big)$. Examples like this may raise the question on whether $\nabla Mf$ belongs to a better $l^p$ space (i.e. $p<1$) when $f \in l^1$. It turns out that the general answer is negative, and Theorem \ref{thm1} is sharp in this sense. To see this consider a function $f \in l^1(\Z)$ such that $f \notin l^p(\Z)$ for any $p<1$, for example $f(n) =1/\big(n \log^2 (n+1)\big)$ for $n\geq 1$, and zero otherwise. Now choose a sequence $1=a_1<a_2 < a_3 < a_4 < ....$ of natural numbers such that 

\begin{itemize}
\item[(i)] $a_2 \geq 4$.
\item[(ii)] $a_{n+1} - a_{n} > a_{n} - a_{n-1} + 2$, for any $n \geq 2$.
\item[(iii)] $f(1) > \frac{\|f\|_{1}}{2(a_2 - a_1) +1}$.
\item[(iv)] $\frac{f(1)}{3} > \frac{\|f\|_{1}}{2(a_2 - a_1-1) +1}$.
\item[(v)] $f(n) > \frac{\|f\|_{1}}{2(a_{n} - a_{n-1}) +1}$, for any $n \geq 2$.
\item[(vi)] $\frac{f(n)}{3} >  \frac{\|f\|_{1}}{2(a_{n} - a_{n-1}+1) +1}$, for any $n \geq 2$.
\end{itemize}
Define the function $g:\Z \to \R$ given by $g(a_n) = f(n)$ for $n\geq 1$, and zero otherwise. Note that $\|g\|_{l^1} = \|f\|_{l^1}$. Conditions (i)-(vi) above guarantee that, for the one-dimensional discrete centered Hardy-Littlewood maximal operator $M$, we have $Mg(a_n) = f(n)$ and $Mg(a_n +1) = \frac{f(n)}{3}$, for $n \geq 1$. Thus $(Mg)'(a_n) =  \frac{2f(n)}{3}$, and thus $(Mg)' \notin l^p(\Z)$ for any $p<1$.

\medskip

\noindent{\it Remark 3:}  Another interesting variant would be to consider the spherical maximal operator \cite{B, S} and its discrete analogue \cite{MSW}. The non-endpoint regularity of the continuous operator in Sobolev spaces was proved in \cite{HO} and it would be interesting to investigate what happens in the endpoint case, both in the continuous and in the discrete settings.

\section{Proof of Theorem \ref{thm1} - Boundedness} \label{boundedness}

\subsection{Centered case} We start with some arithmetic and geometric properties of the sets $\overline{\Omega}_r$. From \eqref{asymp} we can find a constant $c_1$ depending only on the dimension $d$ and the set $\Omega$ such that
\begin{equation}\label{asymp1}
N(\vec{x}_0,r) \leq C_{\Omega} \big(r + c_1\big)^d,
\end{equation}
and 
\begin{equation}\label{asymp2}
N(\vec{x}_0,r) \geq \max \Big\{ C_{\Omega} \big(\max\{r - c_1,0\}\big)^d, 1\Big\} =: C_{\Omega} \big(r - c_1\big)_+^d.
\end{equation}
Over \eqref{asymp2} it should be clear that if $\vec{x}_0 \in \Z^d$ we can take $r\geq 0$, and if $\vec{x}_0 \notin \Z^d$ we shall only be taking radii $r$ so that the corresponding ball contains at least one lattice point to calculate the average. We define $c_2>c_1$ as the constant such that  
\begin{equation*}
C_{\Omega}(c_2-c_1)^d =1.
\end{equation*}
Since $\Omega$ is bounded, there exists $\lambda >0$ (depending only on $\Omega$) such that $\overline{\Omega} \subset \overline{B_{\lambda}}$ (note that $\lambda \geq 1$ since $\vec{e}_d \in\overline{\Omega} $). This means that if $\vec{p} \in \overline{\Omega}_{r}(\vec{x}_0)$ then 
\begin{equation}\label{boundlambda}
\big|\vec{p} - \vec{x}_0\big| \leq \lambda r.
\end{equation}
These constants $c_1$, $c_2$ and $\lambda$ will be fixed throughout the rest of the paper.

\subsubsection{Set up} We want to show that
\begin{equation}\label{sec2.1}
\big\| \nabla Mf \|_{l^1(\Z^d)} \leq C \|f\|_{l^1(\Z^d)}
\end{equation}
for a suitable $C$ that might depend on $d$ and $\Omega$ in principle. We assume without loss of generality that $f\geq 0$. It suffices to prove that 
\begin{equation*}
\left\| \frac{\partial}{\partial x_i} Mf \right\|_{l^1(\Z^d)} \leq \widetilde{C} \|f\|_{l^1(\Z^d)},
\end{equation*}
for any $i = 1,2,...,d$. We will work with $i = d$ (the other cases are analogous). Let us write each $\vec{n} = (n_1, n_2, ..., n_d) \in \Z^d$ as $\vec{n} = (n', n_d)$, where $n' = (n_1, n_2, ..., n_{d-1}) \in \Z^{d-1}$. For each $n' \in \Z^{d-1}$ we will consider the sum over the line perpendicular to $\Z^{d-1}$ passing through $n'$, i.e.
\begin{equation*}
\sum_{l=-\infty}^{\infty}  \left| \frac{\partial}{\partial x_d} Mf \big(n',l\big)\right|= \sum_{l=-\infty}^{\infty} \big| Mf\big(n', l+1\big)  - Mf\big(n',l\big)\big|.
\end{equation*}
For a discrete function $g: \Z \to \R$ we say that a point $a$ is a {\it local maximum} of $g$ if $g(a-1) \leq g(a)$ and $g(a+1) < g(a)$. Analogously, we say that a point $b$ is a {\it local minimum} of $g$ if $g(b-1) \geq g(b)$ and $g(b+1) > g(b)$. We let $\{a_i\}_{i \in \Z}$ and $\{b_i\}_{i \in \Z}$ be the sequences of local maxima and local minima of $Mf\big(n', \cdot\big)$ ordered as follows:
\begin{equation*}
...< b_{-1} < a_{-1} < b_0 < a_0 < b_1 < a_1< ....
\end{equation*}
Observe that this sequence (that depends on $n'$) might be finite (either on one side or both). In this case, since $Mf \in l^1_{weak}(\Z^d)$, it would terminate in a local maximum and minor modifications would have to be done in the argument we present below. For simplicity let us proceed with the case where the sequence of local extrema is infinite on both sides. In this case we have
\begin{equation}\label{sum1}
\sum_{l=-\infty}^{\infty}  \left| \frac{\partial}{\partial x_d} Mf \big(n',l\big)\right| = 2 \sum_{j=-\infty}^{\infty}  \left\{Mf\big(n', a_j\big)  - Mf\big(n',b_j\big)\right\}.
\end{equation}

\subsubsection{The double counting argument} Let $r_j$ be the minimum radius such that the supremum in \eqref{eq1.3} is attained for the point $\big(n', a_j\big)$, i.e.
\begin{equation}\label{average_r_j}
Mf\big(n', a_j\big) = A_{r_j} f\big(n',a_j\big) :=  \frac{1}{N(r_j)} \sum_{\vec{m} \in \overline{\Omega}_{r_j}} f\big(\big(n',a_j\big) + \vec{m}\big).
\end{equation}
If we consider the radius $s_j = r_j + (a_j - b_j)$ centered at the point $\big(n', b_j\big)$ we obtain 
\begin{equation}\label{average_s_j}
Mf\big(n', b_j\big) \geq A_{s_j} f\big(n',b_j\big) =  \frac{1}{N(r_j + (a_j - b_j))} \sum_{\vec{m} \in \overline{\Omega}_{s_j}} f\big(\big(n',b_j\big) + \vec{m}\big).
\end{equation}
The observation that motivates this particular choice of the radius $s_j$ is that $\overline{\Omega}_{r_j}\big(\big(n',a_j\big)\big) \subset \overline{\Omega}_{s_j}\big(\big(n',b_j\big)\big)$, which follows from the convexity of $\overline{\Omega}$ and the fact that $\vec{e}_d \in \partial  \overline{\Omega}$. 

\smallskip

From \eqref{sum1}, \eqref{average_r_j} and \eqref{average_s_j} we obtain
\begin{align}
\begin{split}\label{sum2}
\left\| \frac{\partial}{\partial x_d} Mf \right\|_{l^1(\Z^d)} & = \sum_{n' \in \Z^{d-1}} \sum_{l=-\infty}^{\infty}  \left| \frac{\partial}{\partial x_d} Mf \big(n',l\big)\right| \\
&  \leq \sum_{n' \in \Z^{d-1}} 2 \sum_{j=-\infty}^{\infty}  \left\{A_{r_j}f\big(n', a_j\big)  - A_{s_j}f\big(n',b_j\big)\right\},
\end{split}
\end{align}
where $a_j = a_j(n')$ and $b_j = b_j(n')$. We now consider a general point $\vec{p} = (p_1, p_2,$ $...., p_d) \in \Z^d$, also represented as $\vec{p} = \big(p', p_d\big)$ with $p' \in \Z^{d-1}$. We want to evaluate the maximum contribution that $f\big(p', p_d\big)$ might have to the right-hand side of \eqref{sum2}. For given $n'$ and $j$, this contribution will only be positive if the point $\big(p', p_d\big)$ belongs to both sets $\overline{\Omega}_{r_j}\big(\big(n',a_j\big)\big)$ and $\overline{\Omega}_{s_j}\big(\big(n',b_j\big)\big)$ (in case the point $\big(p', p_d\big)$ belongs only to $\overline{\Omega}_{s_j}\big(\big(n',b_j\big)\big)$ or does not belong to any of these $\Omega$-balls, the contribution is negative or zero and we disregard it). 

\smallskip

Since  $\big(p', p_d\big) \in \overline{\Omega}_{r_j}\big(\big(n',a_j\big)\big)$, from \eqref{boundlambda} we have 
\begin{equation}\label{boundlambda*}
\big|\big(p', p_d\big) - \big(n', a_j\big)\big| \leq \lambda r_j.
\end{equation}
Using \eqref{asymp1},  \eqref{asymp2} and \eqref{boundlambda*}, we can estimate the maximum contribution of $f\big(p', p_d\big)$, for given $n'$ and $j$, on the associated summand on right-hand side of \eqref{sum2} as
\begin{align}\label{bigsum}
\begin{split}
&f\big(p', p_d\big)  \left(  \frac{1}{N(r_j)}  -  \frac{1}{N(r_j + a_j-b_j)}\right)\\
&  \leq  f\big(p', p_d\big)  \left(  \frac{1}{N(r_j)}  -  \frac{1}{N(r_j + a_j-a_{j-1})}\right)\\
& \leq  f\big(p', p_d\big)  \left(  \frac{1}{ C_{\Omega}(r_j - c_1)^d_+}  -  \frac{1}{C_{\Omega} (r_j + a_j - a_{j-1} + c_1)^d}\right)\\
& \leq  f\big(p', p_d\big)  \left(  \frac{1}{C_{\Omega}\left(\lambda^{-1} \big(|p' - n'|^2 + (p_d - a_j)^2\big)^{1/2} - c_1\right)^d_+} \right.\\
&  \ \ \ \ \ \ \ \ \ \ \ \ \ \ \ -  \left. \min\left\{ \frac{1}{C_{\Omega}\left(c_2+a_j - a_{j-1}+ c_1\right)^d},\right.\right.\\
& \ \ \ \ \ \ \ \ \ \ \ \ \ \ \    \left. \left. \frac{1}{C_{\Omega}\left(\lambda^{-1}\big(|p' - n'|^2 + (p_d - a_j)^2\big)^{1/2} +a_j - a_{j-1}+ c_1\right)^d}\right\} \right),
\end{split}
\end{align}
In the last inequality of \eqref{bigsum} we have used \eqref{boundlambda*} and the fact that the function
\begin{equation*}
g(x) =  \left(  \frac{1}{ C_{\Omega} (x - c_1)^d}  -  \frac{1}{C_{\Omega} (x+ a_j - a_{j-1} + c_1)^d}\right)
\end{equation*}
is decreasing as $x \to \infty$, for $x \geq c_2$. If we sum \eqref{bigsum} over all $j$ and then over all $n' \in \Z^{d-1}$ we find an upper bound for the contribution of $f\big(p', p_d\big)$ to the right-hand side of \eqref{sum2}. This is given by
\begin{align}\label{bigsum2}
\begin{split}
& 2 f\big(p', p_d\big) \sum_{n' \in \Z^d} \sum_{j=-\infty}^{\infty} \left(  \frac{1}{C_{\Omega}\left(\lambda^{-1} \big(|p' - n'|^2 + (p_d - a_j)^2\big)^{1/2} - c_1\right)^d_+} \right.\\
&  \ \ \ \ \ \ \ \ \ \ \ \ \ \ \   -  \left. \min\left\{ \frac{1}{C_{\Omega}\left(c_2+a_j - a_{j-1}+ c_1\right)^d},\right.\right.\\
& \ \ \ \ \ \ \ \ \ \ \ \ \ \ \     \left. \left. \frac{1}{C_{\Omega}\left(\lambda^{-1}\big(|p' - n'|^2 + (p_d - a_j)^2\big)^{1/2} +a_j - a_{j-1}+ c_1\right)^d}\right\} \right).
\end{split}
\end{align}

\subsubsection{The summability argument} We now prove that the double sum in \eqref{bigsum2} is bounded {\it independently of the the point $\big(p', p_d\big)$ and the increasing sequence $\{a_j\}$}. For this we may assume $p' = 0$ (since the sum is over all $n' \in \Z^{d-1}$ we can just change variables here to $m' = n' + p'$). We also assume $p_d =0$, since we may consider the increasing sequence $a_j' = a_j + p_d$. The problem becomes then to bound
\begin{align}\label{bigsum3}
\begin{split}
& S(\{a_j\}) = \sum_{n' \in \Z^{d-1}} \sum_{j=-\infty}^{\infty} \left(  \frac{1}{C_{\Omega}\left(\lambda^{-1} \left(|n'|^2 + a_j^2\right)^{1/2} - c_1\right)^d_+} \right.\\
&  \!-  \!\left. \min\!\left\{ \!\!\frac{1}{C_{\Omega}\!\left(c_2\!+\!a_j\! -\! a_{j-1}\!+\! c_1\right)^d}, \frac{1}{C_{\Omega}\!\left(\!\lambda^{-1}\!\left(|n'|^2 + a_j^2\right)^{1/2} \!+\!a_j\! -\! a_{j-1}\!+\! c_1\!\right)^d}\!\!\right\} \!\!\right)
\end{split}
\end{align}
independently of the increasing sequence $\{a_j\}$ of integers. The key tool is the lemma below.

\begin{lemma}[Summability lemma] \label{summability}
For any increasing sequence  $\{a_j\}_{j \in \Z}$ of integers consider the sum $ S(\{a_j\}) $ given by \eqref{bigsum3}. The sum  $ S(\{a_j\}) $ is maximized for the sequence $a_j = j$, and in this case the sum is finite.
\end{lemma}

\begin{proof} Suppose we have two terms in the sequence, say $a_0$ and $a_1$ that are not consecutive. Let us prove that if we introduce a term $\tilde{a}_0$ in the sequence, with $a_0 < \tilde{a}_0 < a_1$, the overall sum does not decrease. For this it is sufficient to see that 
\begin{align*}
\begin{split}
&\left(  \frac{1}{C_{\Omega}\left(\lambda^{-1} \big(|n'|^2 + a_1^2\big)^{1/2} - c_1\right)^d_+} \right.\\
&  \ \ \ \ -  \left. \min\left\{ \frac{1}{C_{\Omega}\!\left(c_2\!+\!a_1\! -\! a_{0}\!+\! c_1\right)^d}, \frac{1}{C_{\Omega}\!\left(\lambda^{-1}\big(|n'|^2 + a_1^2\big)^{1/2} \!+\!a_1\! -\! a_{0}\!+\! c_1\right)^d}\right\} \right)\\
& \leq \left(  \frac{1}{C_{\Omega}\left(\lambda^{-1} \big(|n'|^2 + a_1^2\big)^{1/2} - c_1\right)^d_+} \right.\\
&  \ \ \ \ -  \left. \min\left\{ \frac{1}{C_{\Omega}\!\left(c_2\!+\!a_1\! -\! \tilde{a}_{0}\!+\! c_1\right)^d}, \frac{1}{C_{\Omega}\!\left(\lambda^{-1} \big(|n'|^2 + a_1^2\big)^{1/2} \!+\!a_1\! -\! \tilde{a}_{0}\!+\! c_1\right)^d}\right\} \right)\\
& \ \ + \left(  \frac{1}{C_{\Omega}\left(\lambda^{-1} \big(|n'|^2 + \tilde{a}_0^2\big)^{1/2} - c_1\right)^d_+} \right.\\
&  \ \ \ \ -  \left. \min\left\{ \frac{1}{C_{\Omega}\!\left(c_2\!+\!\tilde{a}_0\! -\! a_{0}\!+\! c_1\right)^d}, \frac{1}{C_{\Omega}\!\left(\lambda^{-1}\big(|n'|^2 + \tilde{a}_0^2\big)^{1/2} \!+\!\tilde{a}_0\! -\! a_{0}\!+\! c_1\right)^d}\right\} \right),
\end{split}
\end{align*}
and this is true if and only if 
\begin{align*}
\begin{split}
&\left(  \min\left\{ \frac{1}{C_{\Omega}\!\left(c_2\!+\!a_1\! -\! \tilde{a}_{0}\!+\! c_1\right)^d}, \frac{1}{C_{\Omega}\!\left(\lambda^{-1}\big(|n'|^2 + a_1^2\big)^{1/2} \!+\!a_1\! -\! \tilde{a}_{0}\!+\! c_1\right)^d}\right\} \right.\\
&  \ \ \ \ -  \left. \min\left\{ \frac{1}{C_{\Omega}\!\left(c_2\!+\!a_1\! -\! a_{0}\!+\! c_1\right)^d}, \frac{1}{C_{\Omega}\!\left(\lambda^{-1}\big(|n'|^2 + a_1^2\big)^{1/2} \!+\!a_1\! -\! a_{0}\!+\! c_1\right)^d}\right\} \right)\\
& \leq  \left(  \frac{1}{C_{\Omega}\left(\lambda^{-1} \big(|n'|^2 + \tilde{a}_0^2\big)^{1/2} - c_1\right)^d_+} \right.\\
&  \ \ \ \ -  \left. \min\left\{ \frac{1}{C_{\Omega}\!\left(c_2\!+\!\tilde{a}_0\! -\! a_{0}\!+\! c_1\right)^d}, \frac{1}{C_{\Omega}\!\left(\lambda^{-1}\big(|n'|^2 + \tilde{a}_0^2\big)^{1/2} \!+\!\tilde{a}_0\! -\! a_{0}\!+\! c_1\right)^d}\right\} \right).
\end{split}
\end{align*}
The last inequality can be verified from the fact that 
\begin{equation*}
g(x) = \frac{1}{C_{\Omega}x^d}  - \frac{1}{C_{\Omega}(x + \tilde{a}_0- a_{0})^d}
\end{equation*}
is decreasing as $x \to \infty$, for $x\geq 0$, and the fact that 
\begin{equation*}
\lambda^{-1}\left(|n'|^2 + a_1^2\right)^{1/2} +\big(a_1 - \tilde{a}_{0}\big)\geq \lambda^{-1}\left(|n'|^2 + \tilde{a}_0^2\right)^{1/2}.
\end{equation*}
The latter follows by calling $a_1 = \tilde{a}_{0} + t$ (note that $t\geq0$), and then differentiating the expression with respect to the variable $t$ to check the sign (here we make use of the fact that $\lambda \geq 1$, since we might have $|\tilde{a}_0| > |a_1|$). 

\smallskip

Therefore the required sum \eqref{bigsum3} is bounded by above by the sum considering the particular sequence $a_j = j$. This gives us 
\begin{align}\label{Ctilde}
\begin{split}
&  S = \sum_{n' \in \Z^{d-1}} \sum_{j=-\infty}^{\infty} \left(  \frac{1}{C_{\Omega}\left(\lambda^{-1} \big(|n'|^2 + j^2\big)^{1/2} - c_1\right)^d_+} \right.\\
& \ \ \ \ \ \ \ \ \ -  \left. \min\left\{ \frac{1}{C_{\Omega}\!\left(c_2\!+1\!+\! c_1\right)^d}, \frac{1}{C_{\Omega}\!\left(\lambda^{-1} \big(|n'|^2 + j^2\big)^{1/2} \!+1\!+\! c_1\right)^d}\right\} \right)\\
& \!\!= \!\!\sum_{\vec{n}\in \Z^d} \!\!\left(\!  \frac{1}{C_{\Omega}\big(\lambda^{-1} |\vec{n}| - c_1\big)^d_+} -  \min\!\left\{\! \frac{1}{C_{\Omega}\!\left(c_2\!+\!1\!+\! c_1\right)^d}\,, \frac{1}{C_{\Omega}\big(\lambda^{-1} |\vec{n}|\!+\!1\!+\! c_1\big)^d}\!\right\}\! \!\right)\\
& \leq \sum_{\lambda^{-1}|\vec{n}| \leq c_2} 1  \ +  \sum_{\lambda^{-1}|\vec{n}|>c_2} \left( \frac{1}{C_{\Omega}\big(\lambda^{-1} |\vec{n}| - c_1\big)^d} -  \frac{1}{C_{\Omega} \big(\lambda^{-1} |\vec{n}|+1+ c_1\big)^d}\right) \\
&= \widetilde{C}(d, \Omega) < \infty.
\end{split}
\end{align}
\end{proof}

\subsubsection{Conclusion}
We have proved that the contribution of a generic point $f(p_1, p_2,$ $..., p_d)$ to  the right-hand side of \eqref{sum2} is at most a constant $2\, \widetilde{C} =2 \,\widetilde{C}(d,\Omega) $ and therefore, when we sum over all points,  we get
\begin{equation*}
\Big\| \frac{\partial}{\partial x_d} Mf \Big\|_{l^1(\Z^d)} \leq 2 \, \widetilde{C} \|f\|_{l^1(\Z^d)}.
\end{equation*}
Since the same holds for any direction we obtain the desired inequality \eqref{sec2.1}.

\subsection{Non-centered case} We will indicate here the basic modifications that have to be made in comparison with the proof for the centered case. The set up is the same up to the beginning of the double counting argument. For a given point $\big(n',a_j\big)$ we can pick a point $\vec{x}_j$ and a radius $r_j$ such that $\big(n', a_j\big) \in \overline{\Omega}_{r_j}(\vec{x}_j)$ and the average over the set $\overline{\Omega}_{r_j}(\vec{x}_j)$ realizes the supremum in the maximal function, i.e.,
\begin{equation}\label{average_r_j_x_j}
\widetilde{M}f\big(n', a_j\big) = A_{(\vec{x}_j, r_j)} f\big(n',a_j\big) :=  \frac{1}{N(\vec{x}_j,r_j)} \sum_{\vec{m} \in \overline{\Omega}_{r_j}(\vec{x}_j)} f(\vec{m}).
\end{equation}
This is guaranteed since any maximizing sequence $\big(\vec{x}_j^k, r_j^k\big)$ of the right-hand side of \eqref{average_r_j_x_j} must be stationary. In fact, we should have the sequence $\big( \vec{x}_j^k, r_j^k\big)$ trapped in a bounded subset $\big|\vec{x}_j^k\big|\leq R$ and $r_j^k \leq R$, for some $R>0$ (since $f \in l^1(\Z^d)$), and then we would have only a finite number of subsets of $\Z^d$ to choose from for the sum in \eqref{average_r_j_x_j}.

\smallskip

We now consider the $\Omega$-ball of radius $s_j = r_j + a_j - b_j$ centered at $\vec{y}_j = \vec{x}_j - (a_j - b_j)\vec{e}_d$. Note that $(n', b_j) \in \overline{\Omega}_{r_j}(\vec{y}_j) \subset \overline{\Omega}_{s_j}(\vec{y}_j) $. From the convexity of $\overline{\Omega}$ and the fact that $\vec{e}_d \in \partial \overline{\Omega}$ we also have $\overline{\Omega}_{r_j}(\vec{x}_j) \subset \overline{\Omega}_{s_j}(\vec{y}_j)$. Therefore
\begin{equation}\label{average_s_j_y_j}
\widetilde{M}f\big(n', b_j\big) \geq A_{(\vec{y}_j,s_j)} f\big(n',b_j\big) =  \frac{1}{N(\vec{y}_j,s_j)} \sum_{\vec{m} \in \overline{\Omega}_{s_j}(\vec{y}_j)} f(\vec{m}),
\end{equation}
and 
\begin{align}
\begin{split}\label{sum2nc}
\left\| \frac{\partial}{\partial x_d} \widetilde{M}f \right\|_{l^1(\Z^d)} & = \sum_{n' \in \Z^{d-1}} \sum_{l=-\infty}^{\infty}  \left| \frac{\partial}{\partial x_d} \widetilde{M}f \big(n',l\big)\right| \\
&  \leq \sum_{n' \in \Z^{d-1}} 2 \sum_{j=-\infty}^{\infty}  \left\{A_{( \vec{x}_j,r_j)}f\big(n', a_j\big)  - A_{(\vec{y}_j,s_j)}f\big(n',b_j\big)\right\}.
\end{split}
\end{align}

\smallskip

Consider a point $\vec{p} = \big(p', p_d\big) \in \Z^d$. The term $f\big(p', p_d\big)$ will only contribute positively to a summand on the right-hand side of \eqref{sum2nc} if $\big(p', p_d\big) \in \overline{\Omega}_{r_j}(\vec{x}_j)$. In this case, since $\big(n', a_j\big)  \in \overline{\Omega}_{r_j}(\vec{x}_j)$, using \eqref{boundlambda} we have 
\begin{equation}\label{nccond}
\big|\big(p', p_d\big) - (n', a_j)\big| \leq 2\, \lambda \, r_j.
\end{equation}
The rest of the proof is the same.

\section{Proof of Theorem \ref{thm1} - Continuity} 

\subsection{Centered case} We want to show that if $f_k \to f$ in $l^1(\Z^d)$ then $\nabla Mf_k \to \nabla M f$ in $l^1(\Z^d)$. 

\subsubsection{Set up} Since $\big||f_k|-|f|\big| \leq \big|f_k - f\big|$ and the maximal operator only sees the absolute value we may assume without loss of generality that $f_k \geq 0$ for all $k$, and that $f\geq0$. It suffices to prove the result for each partial derivative, i.e. that
\begin{equation}\label{limit}
\left\|\frac{\partial}{\partial x_i}Mf_k - \frac{\partial}{\partial x_i}Mf\right\|_{l^1(\Z^d)} \to 0
\end{equation}
as $k \to \infty$, for each $i=1,2,...,d$. We shall prove it for $i=d$ and the other cases are analogous. 

\smallskip

\subsubsection{A discrete version of Luiro's lemma}
For a function $g \in l^1(\Z^d)$ and a point $\vec{n} \in \Z^d$ let us define $\mc{R}g(\vec{n})$ as the set of all radii that realize the supremum in the maximal function at the point $\vec{n}$, i.e.
\begin{equation*}
\mc{R}g(\vec{n}) = \left\{r \in [0,\infty);\ Mg(\vec{n}) = A_{r} |g|(\vec{n}) =  \frac{1}{N(r)} \sum_{\vec{m} \in \overline{\Omega}_{r}}\big|g\big(\vec{n} + \vec{m}\big)\big|\right\}.
\end{equation*}
The next lemma gives us information about the convergence of these sets of radii. It can be seen as the discrete analogue of \cite[Lemma 2.2] {Lu1}.
\begin{lemma}\label{radii}
Let $f_k \to f $ in  $l^1(\Z^d)$. Given $R>0$ there exists $k_0= k_0(R)$ such that, for $k\geq k_0$, we have $\mc{R}f_k(\vec{n}) \subset \mc{R}f(\vec{n})$ for each $\vec{n} \in  \overline{B_R}$.
\end{lemma}
\begin{proof}
Fix $\vec{n} \in \overline{B_R}$ and consider the application $r\mapsto A_r f(\vec{n})$ for $r\geq 0$. From the fact that $f \in l^1(\Z^d)$ together with \eqref{asymp2} we can see that $A_r f(\vec{n}) \to 0$ as $r \to \infty$. Therefore the set of values in the image $\{A_r f(\vec{n}); \ r\geq0\}$ such that $A_r f(\vec{n}) \geq \hh Mf(\vec{n})$ is a finite set. There exists then a ``second larger" value which falls short of the maximum by a quantity we define as $\epsilon(\vec{n})$, i.e. if $A_r f(\vec{n}) > Mf(\vec{n}) - \epsilon(\vec{n})$ then $A_r f(\vec{n}) = Mf(\vec{n})$ and $r\in \mc{R}f(\vec{n})$. Define
\begin{equation*}
\epsilon = \frac{1}{3} \min \big\{\epsilon(\vec{n});\  \vec{n} \in \overline{B_R} \,\big\}.
\end{equation*}
Since $f_k \to f $ in  $l^1(\Z^d)$, we have $f_k \to f $ in  $l^{\infty}(\Z^d)$. Pick $k_0$ such that for $k \geq k_0$ we have $\|f_k - f\|_{l^{\infty}} \leq \epsilon$. For any $\vec{n} \in  \overline{B_R}$ if we take $s \in \mc{R}f(\vec{n})$ we have
\begin{equation}\label{sec3.1}
Mf(\vec{n}) = A_s f (\vec{n}) = A_s f_k (\vec{n}) + A_s(f-f_k)(\vec{n})   \leq Mf_k(\vec{n}) + \epsilon.
\end{equation}
Now given $r_k \in \mc{R}f_k(\vec{n})$ we can use \eqref{sec3.1} to obtain
\begin{align*}
\begin{split}
A_{r_k}f(\vec{n}) &= A_{r_k}f_k(\vec{n}) + A_{r_k}(f-f_k)(\vec{n})\\ 
& = Mf_k (\vec{n}) + A_{r_k}(f-f_k)(\vec{n}) \geq Mf_k (\vec{n}) - \epsilon \geq Mf(\vec{n}) -2\epsilon\,,
\end{split}
\end{align*}
and from the definition of $\epsilon$ and $\epsilon(\vec{n})$ we conclude that $r_k \in \mc{R}f(\vec{n})$.
\end{proof}

\subsubsection{Reduction via the Brezis--Lieb lemma}
Given $\epsilon >0$, we can find $k_0$ such that $\|f_k - f\|_{l^{\infty}} \leq \epsilon$, and using Lemma \ref{radii} for a fixed $\vec{n}\in \Z^d$, we can choose $k_1 \geq k_0$ so that we also have $\mc{R}f_k(\vec{n}) \subset \mc{R}f(\vec{n})$ for $k \geq k_1$. Taking any $r_k \in \mc{R}f_k(\vec{n})$ we have
\begin{equation}\label{equat_k}
\big|Mf(\vec{n}) - Mf_k(\vec{n})\big| = \big|A_{r_k}f(\vec{n}) - A_{r_k}f_k(\vec{n}) \big| \leq \epsilon\,,
\end{equation}
for $k \geq k_1$ and thus $Mf_k(\vec{n}) \to Mf(\vec{n})$ as $k\to \infty$. The same can be said replacing $\vec{n}$ by $\vec{n} + \vec{e}_d$ and thus we find that 
\begin{equation}\label{conv1}
\frac{\partial}{\partial x_d}Mf_k (\vec{n}) \to \frac{\partial}{\partial x_d}Mf (\vec{n})
\end{equation}
pointwise as $k \to \infty$. Since 
\begin{equation*}
\left| \Big|\frac{\partial}{\partial x_d}Mf_k(\vec{n}) \Big| - \Big|\frac{\partial}{\partial x_d}Mf_k(\vec{n} ) - \frac{\partial}{\partial x_d}Mf(\vec{n}) \Big|\right| \leq \left|\frac{\partial}{\partial x_d}Mf(\vec{n}) \right| 
\end{equation*}
and the latter is in $l^1(\Z^d)$ from the boundedness part of the theorem, an application of the dominated convergence theorem with \eqref{conv1} gives us
\begin{equation*}
\lim_{k\to \infty} \left\{\left\|\frac{\partial}{\partial x_d}Mf_k \right\|_{l^1(\Z^d)} - \left\|\frac{\partial}{\partial x_d}Mf_k - \frac{\partial}{\partial x_d}Mf \right\|_{l^1(\Z^d)} \right\} = \left\|\frac{\partial}{\partial x_d}Mf\right\|_{l^1(\Z^d)}. 
\end{equation*}
Therefore, to prove \eqref{limit} it suffices to show that
\begin{equation}\label{key}
\lim_{k\to \infty} \left\|\frac{\partial}{\partial x_d}Mf_k \right\|_{l^1(\Z^d)}  = \left\|\frac{\partial}{\partial x_d}Mf\right\|_{l^1(\Z^d)}. 
\end{equation}
The reduction to \eqref{key} is the content of the Brezis-Lieb lemma \cite{BL} in the case $p=1$. We henceforth focus our efforts in proving \eqref{key}.

\subsubsection{Lower bound} From Fatou's lemma and \eqref{conv1} we have
\begin{equation}\label{liminf}
 \left\|\frac{\partial}{\partial x_d}Mf\right\|_{l^1(\Z^d)} \leq \liminf_{k\to \infty} \left\|\frac{\partial}{\partial x_d}Mf_k \right\|_{l^1(\Z^d)} .
\end{equation}

\subsubsection{Upper bound} Given $\epsilon >0$ we shall prove that there exists $k_0 = k_0(\epsilon)$ such that for $k \geq k_0$ we have
\begin{equation}\label{limsup1}
\left\|\frac{\partial}{\partial x_d}Mf_k \right\|_{l^1(\Z^d)} \leq  \left\|\frac{\partial}{\partial x_d}Mf\right\|_{l^1(\Z^d)} + \ \epsilon.
\end{equation}
This would imply that 
\begin{equation*}
\limsup_{k\to \infty} \left\|\frac{\partial}{\partial x_d}Mf_k \right\|_{l^1(\Z^d)} \leq  \left\|\frac{\partial}{\partial x_d}Mf\right\|_{l^1(\Z^d)},
\end{equation*}
which together with \eqref{liminf} would prove that the limit exists and \eqref{key} holds.

\smallskip

Let us start with a sufficiently large integer radius $R$ (to be properly chosen later) and consider the cube $ \big\{ \vec{x} \in \R^d; |\vec{x}|_{\infty} \leq 2R\big\}$. Let us continue writing $\vec{n} \in \Z^d$ as $\vec{n} = \big(n', n_d\big)$ with $n' \in \Z^{d-1}$. We write the required sum in the following way
\begin{align}\label{final0}
\begin{split}
 \left\|\frac{\partial}{\partial x_d}Mf_k \right\|_{l^1(\Z^d)}  \!& \!=\! \sum_{\stackrel{|n'|_{\infty} \leq 2R}{|n_d|_{\infty} \leq 2R }} \left| \frac{\partial}{\partial x_d}Mf_k \big(n',n_d\big)\right| + \! \sum_{\stackrel{|n'|_{\infty} > 2R}{n_d \in \Z}} \left| \frac{\partial}{\partial x_d}Mf_k\big(n',n_d\big)\right|\\
 & \ \ \ \ \ \ \ \ \ \ \ \ \ \ \ \ \ \ \ \ \ \ + \sum_{\stackrel{|n'|_{\infty} \leq 2R}{|n_d|_{\infty} > 2R }} \left| \frac{\partial}{\partial x_d}Mf_k \big(n',n_d\big)\right|\\
 &:= S_1 + S_2 + S_3.
 \end{split}
\end{align}
We shall bound $S_1$, $S_2$ and $S_3$ separately.

\subsubsection{Bound for $S_1$} Let us pick $\epsilon_1>0$ (to be properly chosen later). With the aid of Lemma \ref{radii} we find $k_1 = k_1(\epsilon_1, R)$ such that $\mc{R}f_k(\vec{n}) \subset \mc{R}f(\vec{n})$ for each $\vec{n}$ with $|\vec{n}|_{\infty} \leq 2R+1$ and 
\begin{equation}\label{k1}
\|f_k - f\|_{l^{\infty}(\Z^d)} \leq \epsilon_1,
\end{equation}
for $k \geq k_1$. Using \eqref{equat_k} we have that
\begin{equation*}
\left| \frac{\partial}{\partial x_d}Mf_k (\vec{n}) - \frac{\partial}{\partial x_d}Mf (\vec{n}) \right| \leq 2\epsilon_1,
\end{equation*}
 for any $\vec{n}$ with $|\vec{n}|_{\infty} \leq 2R$. Thus
\begin{align}\label{final1}
\begin{split}
S_1 & =\!\! \!\!\sum_{\stackrel{|n'|_{\infty} \leq 2R}{|n_d|_{\infty} \leq 2R }} \left| \frac{\partial}{\partial x_d}Mf_k \big(n',n_d\big)\right| \!\leq \!\!\!\! \sum_{\stackrel{|n'|_{\infty} \leq 2R}{|n_d|_{\infty} \leq 2R }} \left| \frac{\partial}{\partial x_d}Mf \big(n',n_d\big)\right| \!+ 2\,\epsilon_1 (4R+1)^d\\
&  \ \ \ \ \ \ \ \ \ \ \ \ \ \ \ \ \ \ \ \ \ \ \ \ \ \leq  \left\|\frac{\partial}{\partial x_d}Mf\right\|_{l^1(\Z^d)} +  2\,\epsilon_1 (4R+1)^d.
\end{split}
\end{align}

\subsubsection{Bound for $S_2$}  Here we start with the same idea (and notation for the local maxima and local minima over vertical lines) as in \eqref{sum2}
\begin{align}
\begin{split}\label{sum2_cont}
S_2 & = \sum_{|n'|_{\infty} >2R}\  \sum_{l=-\infty}^{\infty}  \Big| \frac{\partial}{\partial x_d} Mf_k \big(n',l\big)\Big| \\
&  \leq \sum_{|n'|_{\infty} >2R}  2 \sum_{j=-\infty}^{\infty}  \left\{A_{r_j}f_k\big(n', a_j\big)  - A_{s_j}f_k\big(n',b_j\big)\right\}.
\end{split}
\end{align}
We find an upper bound for the contribution of a generic point $f_k\big(p', p_d\big)$ to the right-hand side of \eqref{sum2_cont} as previously done in \eqref{bigsum2}. This is given by 
\begin{align}\label{bigsum2_cont}
\begin{split}
& 2 f_k\big(p', p_d\big) \sum_{|n'|_{\infty} >2R}\  \sum_{j=-\infty}^{\infty} \left(  \frac{1}{C_{\Omega}\left(\lambda^{-1} \big(|p' - n'|^2 + (p_d - a_j)^2\big)^{1/2} - c_1\right)^d_+} \right.\\
&  \ \ \ \ \ \ \ \ \ \ \ \ \ \ \   -  \left. \min\left\{ \frac{1}{C_{\Omega}\left(c_2+a_j - a_{j-1}+ c_1\right)^d},\right.\right.\\
& \ \ \ \ \ \ \ \ \ \ \ \ \ \ \     \left. \left. \frac{1}{C_{\Omega}\left(\lambda^{-1}\big(|p' - n'|^2 + (p_d - a_j)^2\big)^{1/2} +a_j - a_{j-1}+ c_1\right)^d}\right\} \right).
\end{split}
\end{align}
Using Lemma \ref{summability} we see that the sum on the right-hand side of \eqref{bigsum2_cont} is majorized by the sum with the sequence $a_j = j$. This gives us
\begin{align}\label{bigsum3_cont}
\begin{split}
& 2 f_k\big(p', p_d\big) \sum_{|n'|_{\infty} >2R} \ \sum_{j=-\infty}^{\infty} \left(  \frac{1}{C_{\Omega}\left(\lambda^{-1} \big(|p' - n'|^2 + j^2\big)^{1/2} - c_1\right)^d_+} \right.\\
&  \ \ \ \ \ \ \ \ \ \ \ \ \ \ \   -  \left. \min\left\{ \frac{1}{C_{\Omega}\left(c_2+1+ c_1\right)^d},\right.\right.\\
& \ \ \ \ \ \ \ \ \ \ \ \ \ \ \     \left. \left. \frac{1}{C_{\Omega}\left(\lambda^{-1}\big(|p' - n'|^2 + j^2\big)^{1/2} +1 + c_1\right)^d}\right\} \right).
\end{split}
\end{align}
We now evaluate this contribution in two distinct sets. Firstly, we consider the case when $\big(p', p_d\big) \in \overline{B_{R}}$, for which we have $|p' - n'| \geq R$. Imposing the condition that
\begin{equation}\label{R-cond}
\lambda^{-1} R > c_2
\end{equation}
we can ensure that the contribution of $f_k\big(p', p_d\big)$ is majorized by
\begin{align}\label{bigsum4_cont}
\begin{split}
& 2 f_k\big(p', p_d\big) \sum_{|\vec{n}| \geq R} \left(  \frac{1}{C_{\Omega}\left(\lambda^{-1} |\vec{n}| - c_1\right)^d} - \frac{1}{C_{\Omega}\left(\lambda^{-1}|\vec{n}| +1 + c_1\right)^d} \right)\\
& \  \ \ \ \ := 2 \,f_k\big(p', p_d\big)\,  h(R).
\end{split}
\end{align}
The fact that $h(R) \to 0$ as $R \to \infty$ is a crucial point in this proof and shall be used when we choose $R$ at the end. Secondly, when $\big(p', p_d\big) \notin \overline{B_{R}}$ the contribution will simply be bounded by 2 $\widetilde{C} f_k\big(p', p_d\big)$ as we found in \eqref{Ctilde}. If we then sum up these contributions and plug them in on the right-hand side of \eqref{sum2_cont} we find
\begin{equation}\label{final2}
S_2 \leq 2 \,h(R)\, \|\chi_{\overline{B_R}} f_k\|_{l^1(\Z^d)} + 2\, \widetilde{C} \,\|\chi_{\overline{B_R}^c} f_k\|_{l^1(\Z^d)}.
\end{equation}

\subsubsection{Bound for $S_3$} We start by noting that 
\begin{align*}
\begin{split}\label{sum2_cont_1}
S_3 & = \sum_{|n'|_{\infty} \leq 2R}\  \sum_{l=2R+1}^{\infty}  \left| \frac{\partial}{\partial x_d} Mf_k \big(n',l\big) \right|+  \sum_{|n'|_{\infty} \leq 2R} \sum_{l=-\infty}^{-2R-1}  \left| \frac{\partial}{\partial x_d} Mf_k \big(n',l\big) \right|\\
& := S_{3}^+ + S_{3}^-.
\end{split}
\end{align*}
Let us provide an upper bound for $S_3^+$. The upper bound for  $S_3^-$ is analogous. We consider the sequence of local maxima $\{a_j\}$ and local minima $\{b_j\}$ for $Mf_k \big(n',l\big)$ when $l \geq 2R+1$. In this situation we do have a first local maximum $a_1$ (which might be the endpoint $2R+1$) and we order this sequence as follows:
\begin{equation*}
2R+1\leq a_1 < b_2 < a_2< b_3< a_3...
\end{equation*}
If the sequence terminates, it will be in a local maximum since $Mf_k \in l^1_{weak}(\Z^d)$, and we can just truncate the sum in the argument below. Keeping the notation as before (and including for convenience $a_0 = b_1 = -\infty$) we have
\begin{equation}\label{sum_2_cont_S3}
S_3^+   \leq \sum_{|n'|_{\infty} \leq 2R}  2 \sum_{j=1}^{\infty}  \left\{A_{r_j}f_k\big(n', a_j\big)  - A_{s_j}f_k\big(n',b_j\big)\right\}.
\end{equation}
The contribution of a generic point $f_k\big(p', p_d\big)$ to the right-hand side of \eqref{sum_2_cont_S3} (following the calculation \eqref{bigsum2}), has an upper bound of
\begin{align}\label{bigsum2_cont_S3}
\begin{split}
& 2 f_k\big(p', p_d\big) \sum_{|n'|_{\infty} \leq 2R} \  \sum_{j=1}^{\infty} \left(  \frac{1}{C_{\Omega}\left(\lambda^{-1} \big(|p' - n'|^2 + (p_d - a_j)^2\big)^{1/2} - c_1\right)^d_+} \right.\\
&  \ \ \ \ \ \ \ \ \ \ \ \ \ \ \   -  \left. \min\left\{ \frac{1}{C_{\Omega}\left(c_2+a_j - a_{j-1}+ c_1\right)^d},\right.\right.\\
& \ \ \ \ \ \ \ \ \ \ \ \ \ \ \     \left. \left. \frac{1}{C_{\Omega}\left(\lambda^{-1}\big(|p' - n'|^2 + (p_d - a_j)^2\big)^{1/2} +a_j - a_{j-1}+ c_1\right)^d}\right\} \right).
\end{split}
\end{align}
Following the ideas of Lemma \ref{summability}, keeping the constraint that $a_0 = -\infty$, the sum on the right-hand side of \eqref{bigsum2_cont_S3} is maximized when $a_j = 2R+j$ for $j\geq1$. We would then have the upper bound
\begin{align}\label{bigsum3_cont_S3}
\begin{split}
&  2 f_k\big(p', p_d\big) \sum_{|n'|_{\infty} \leq 2R}  \frac{1}{C_{\Omega}\left(\lambda^{-1} \big(|p' - n'|^2 + (p_d - 2R-1)^2\big)^{1/2} - c_1\right)^d_+} \\
&+  2 f_k\big(p', p_d\big) \sum_{|n'|_{\infty} \leq 2R} \sum_{j=2}^{\infty} \left(  \frac{1}{C_{\Omega}\left(\lambda^{-1} \big(|p' - n'|^2 + (p_d - 2R-j)^2\big)^{1/2} - c_1\right)^d_+} \right.\\
&  \ \ \ \ \ \ \ \ \ \ \ \ \ \ \   -  \left. \min\left\{ \frac{1}{C_{\Omega}\left(c_2+1+ c_1\right)^d},\right.\right.\\
& \ \ \ \ \ \ \ \ \ \ \ \ \ \ \     \left. \left. \frac{1}{C_{\Omega}\left(\lambda^{-1}\big(|p' - n'|^2 + (p_d - 2R -j)^2\big)^{1/2} +1+ c_1\right)^d}\right\} \right).
\end{split}
\end{align}
Again, we evaluate this contribution separately for $\big(p', p_d\big)$ in the sets $\overline{B_R}$ and $\overline{B_R}^c$. In the first case, if $\big(p', p_d\big) \in \overline{B_R}$ we have $|p_d - 2R -1| \geq R$, and if we choose $R$ satisfying \eqref{R-cond} the contribution of $f_k\big(p', p_d\big)$ will be less than or equal to 
\begin{align}\label{bigsum4_cont_S3}
\begin{split}
&2 f_k\big(p', p_d\big) \left\{\frac{(4R+1)^{d-1}}{C_{\Omega}\left(\lambda^{-1} R-c_1\right)^d} \right.\\
&  \ \ \ \  \ \ \ \ \ \ \ \ \ \ \ \ \ \ \ \ \ \left. +\sum_{|\vec{n}|\geq R}  \left( \frac{1}{C_{\Omega}\left(\lambda^{-1} |\vec{n}|- c_1\right)^d} - \frac{1}{C_{\Omega}\left(\lambda^{-1}|\vec{n}| +1+ c_1\right)^d} \right)\right\}\\
& = 2 f_k\big(p', p_d\big) \left\{\frac{(4R+1)^{d-1}}{C_{\Omega}\left(\lambda^{-1} R-c_1\right)^d} + h(R)\right\}.
\end{split}
\end{align}
In the second case, if $\big(p', p_d\big) \in \overline{B_R}^c$, we just bound the contribution of $f_k\big(p', p_d\big)$ by $2\,\widetilde{C}\, f_k\big(p', p_d\big)$ as in \eqref{Ctilde}. Plugging these upper bounds in \eqref{sum_2_cont_S3} we find
\begin{equation}\label{final3}
S_3^+ \leq  2\,\left\{\frac{(4R+1)^{d-1}}{C_{\Omega}\left(\lambda^{-1} R-c_1\right)^d} + h(R)\right\} \|\chi_{\overline{B_R}} f_k\|_{l^1(\Z^d)} + 2 \,\widetilde{C}\, \|\chi_{\overline{B_R}^c} f_k\|_{l^1(\Z^d)}.
\end{equation}
By symmetry the same bound holds for $S_3^-$.

\subsubsection{Conclusion} Putting together \eqref{final0}, \eqref{final1}, \eqref{final2} and \eqref{final3} we obtain
\begin{align}\label{fim0}
\begin{split}
 & \left\|\frac{\partial}{\partial x_d}Mf_k \right\|_{l^1(\Z^d)} \leq \left\|\frac{\partial}{\partial x_d}Mf\right\|_{l^1(\Z^d)} +  2\,\epsilon_1 (4R+1)^d \\
 &  \ \ \ \ + \left\{4\frac{(4R+1)^{d-1}}{C_{\Omega}\left(\lambda^{-1} R-c_1\right)^d} + 6\, h(R)\right\} \|\chi_{\overline{B_R}} f_k\|_{l^1(\Z^d)} +  6\,\widetilde{C} \|\chi_{\overline{B_R}^c} f_k\|_{l^1(\Z^d)}.
 \end{split}
\end{align}
We choose (in this order) $R$ large enough so that it satisfies \eqref{R-cond}, 
\begin{equation}\label{fim1}
\left\{4\frac{(4R+1)^{d-1}}{C_{\Omega}\left(\lambda^{-1} R-c_1\right)^d} + 6\, h(R)\right\} \leq \frac{\epsilon}{3\big(\|f\|_{l^1(\Z^d)} + 1\big)},
\end{equation}
and 
\begin{equation*}
\|\chi_{\overline{B_R}^c} f\|_{l^1(\Z^d)} \leq \frac{\epsilon}{36\widetilde{C}}.
\end{equation*}
Then we choose $\epsilon_1$ such that 
\begin{equation}\label{fim2}
\epsilon_1 \leq \frac{\epsilon}{6 (4R+1)^d}\,,
\end{equation}
and this generates a $k_1$ as described in \eqref{k1}. We now choose $k_0 \geq k_1$ such that for all $k \geq k_0$ we have 
\begin{equation*}
\|f_k - f\|_{l^1(\Z^d)} \leq \min\left \{\frac{\epsilon}{36\widetilde{C}}\,, \,1\right\},
\end{equation*}
which then implies that 
\begin{equation}\label{fim3}
\|\chi_{\overline{B_R}} f_k\|_{l^1(\Z^d)} \leq \|f_k\|_{l^1(\Z^d)} \leq\big( \|f\|_{l^1(\Z^d)}  +1\big)
\end{equation}
and
\begin{align}\label{fim4}
\begin{split}
\|\chi_{\overline{B_R}^c} f_k\|_{l^1(\Z^d)}  \leq \|\chi_{\overline{B_R}^c} f\|_{l^1(\Z^d)} + \|\chi_{\overline{B_R}^c} (f_k-f)\|_{l^1(\Z^d)} \leq \frac{\epsilon}{18\widetilde{C}}.
\end{split}
\end{align}
Plugging \eqref{fim1}, \eqref{fim2}, \eqref{fim3} and \eqref{fim4} into \eqref{fim0} gives us
\begin{equation*}
\left\|\frac{\partial}{\partial x_d}Mf_k \right\|_{l^1(\Z^d)} \leq \left\|\frac{\partial}{\partial x_d}Mf\right\|_{l^1(\Z^d)} +  \frac{\epsilon}{3} + \frac{\epsilon}{3} + \frac{\epsilon}{3}\,,
\end{equation*}
for all $k \geq k_0$, and the proof is now complete.

\subsection{Non-centered case} We will indicate here the basic changes that have to be made in comparison with the centered case argument.

\smallskip 

For a function $g \in l^1(\Z^d)$ and a point $\vec{n} \in \Z^d$, let us define the set $\widetilde{\mc{R}}g (\vec{n})$ as the set of all pairs $(\vec{x}, r) \in \R^d\times \R^+$ such that $\vec{n} \in \overline{\Omega}_r(\vec{x})$ and the supremum in the non-centered maximal function at $\vec{n}$ is attained for $\overline{\Omega}_r(\vec{x})$, i.e.
\begin{equation*}
\widetilde{\mc{R}}g(\vec{n}) = \left\{ (\vec{x}, r) \in \R^d\times \R^+;\ \widetilde{M}g(\vec{n}) = A_{(\vec{x},r)} |g|(\vec{n}) =  \frac{1}{N(\vec{x},r)} \sum_{\vec{m} \in \overline{\Omega}_{r}(\vec{x})}\big|g\big(\vec{m}\big)\big|\right\}.
\end{equation*}
The proof of the following result is essentially the same as in Lemma \ref{radii}.
\begin{lemma}\label{radii2}
Let $f_k \to f $ in  $l^1(\Z^d)$. Given $R>0$ there exists $k_0= k_0(R)$ such that, for $k\geq k_0$, we have $\widetilde{\mc{R}}f_k(\vec{n}) \subset \widetilde{\mc{R}}f(\vec{n})$ for each $\vec{n} \in  \overline{B_R}$.
\end{lemma}
The rest of the proof is also similar, using  \eqref{average_s_j_y_j}, \eqref{sum2nc} and \eqref{nccond} in the appropriate places.

\section{Acknowledgements}
The first author acknowledges support from CNPq-Brazil grants $473152/2011-8$  and $302809/2011-2$. The second author acknowldges support from NSF grant DMS$-0901040$. We would like to thank Carlos Cabrelli and Ursula Molter for discussions related to Remark 2.


\bibliographystyle{amsplain}

\end{document}